\documentclass[11pt,twoside,reqno]{amsart}
\allowdisplaybreaks
\usepackage{amsmath,amstext,amssymb,epsfig,multicol,enumerate}
\usepackage{graphicx}
\usepackage{mathrsfs}
\calclayout
\makeatletter
\renewcommand{\@seccntformat}[1]{\bf\csname the#1\endcsname.}
\renewcommand{\section}{\@startsection{section}{1}
	\z@{.7\linespacing\@plus\linespacing}{.5\linespacing}
	{\normalfont\upshape\bfseries\centering}}
\renewcommand{\@biblabel}[1]{\@ifnotempty{#1}{#1.}}
\makeatother
\theoremstyle{plain}
\newtheorem{thm}{Theorem}[section]
\newtheorem{lem}[thm]{Lemma}
\newtheorem{prop}[thm]{Proposition}

\theoremstyle{definition}
\newtheorem{ex}[thm]{Example}
\newtheorem{defn}[thm]{Definition}

\usepackage{cancel}
\usepackage[parfill]{parskip}
\usepackage[german]{varioref}
\usepackage[all]{xy}
\usepackage{color}

\def \>{\succ}
\def \<{\prec}

\begin{document}	
\title[Bouzid Mosbahi\textsuperscript{1}, Imed Basdouri\textsuperscript{2}, Jean Lerbet\textsuperscript{3}\textsuperscript{*}]{Cohomology and deformation theory of Averaging Leibniz algebras}
	\author{Bouzid Mosbahi\textsuperscript{1}, Imed Basdouri\textsuperscript{2}, Jean Lerbet\textsuperscript{3}\textsuperscript{*}}
  \address{\textsuperscript{1}Department of Mathematics, Faculty of Sciences, University of Sfax, Sfax, Tunisia}
  \address{\textsuperscript{2}Department of Mathematics, Faculty of Sciences, University of Gafsa, Gafsa, Tunisia}
  \address{\textsuperscript{3}\textsuperscript{*}Laboratoire de Mathématiques et Modélisation d’Évry (UMR 8071) Université d’Évry Val d’Essonne I.B.G.B.I., 23 Bd. de France, 91037 Évry Cedex, France}

\email{\textsuperscript{1}mosbahi.bouzid.etud@fss.usf.tn}
\email{\textsuperscript{2}basdourimed@yahoo.fr}
\email{\textsuperscript{3}\textsuperscript{*}jean.lerbet@univ-evry.fr \\
\textsuperscript{*}Corresponding author}

	
	
	\date{\today}

	\begin{abstract}
In this paper, we introduce the concepts of representation and dual representation for averaging Leibniz algebras. We also develop a cohomology theory for these algebras. Additionally, we explore the infinitesimal and formal deformation theories of averaging Leibniz algebras, showing that the cohomology we define is closely connected to deformation cohomology.
\end{abstract}

\maketitle
\section{ Introduction}\label{introduction}
In 1965, Blokh introduced Leibniz algebras as a generalization of Lie algebras \cite{1}. Later, in 1993, Loday studied them more deeply \cite{2} and, together with Pirashvili, developed their cohomology theory \cite{22}. A \textit{Leibniz algebra} is a vector space $\mathfrak{g}$ over a field $\mathbb{K}$, with a bilinear operation $[ \cdot , \cdot ]_\mathfrak{g} : \mathfrak{g} \otimes \mathfrak{g} \to \mathfrak{g}$ satisfying the (left) \textit{Leibniz identity}:
\begin{align*}
[u, [v, w]_\mathfrak{g}]_\mathfrak{g} = [[u, v]_\mathfrak{g}, w]_\mathfrak{g} + [v, [u, w]_\mathfrak{g}]_\mathfrak{g}, \quad \text{for all} \; u, v, w \in \mathfrak{g}.
\end{align*}

The classification of algebras up to isomorphism is a fundamental problem that helps understand their structure. This has been pursued for low-dimensional Leibniz algebras \cite{12,13} and their generalizations such as Hom-algebras and BiHom-algebras \cite{3,4,5,23,24,25, 26}. Additionally, the study of derivations and centroids provides insight into the inner and outer structures of these algebras \cite{14,15}.

Leibniz algebras have attracted considerable attention in modern algebra, with research focusing on their cohomology, structure, and classification \cite{6,7,8,9,10,11}. In particular, the cohomology theory of Leibniz algebras, introduced by Loday and Pirashvili \cite{22}, has been a key tool in understanding their deformations and extensions.

Recently, there has been growing interest in applications of Leibniz algebras, particularly in the study of \textit{Averaging operators} \cite{16,17}. Averaging operators were first introduced by Reynolds in 1895 in connection with turbulence theory \cite{18}. These operators are closely related to Rota-Baxter, Nijenhuis, and Reynolds operators, which generalize integral-type operators in various algebraic frameworks \cite{19,20,21}. They also appear in recent studies involving groups, racks, and general algebraic systems \cite{16}.

An \textit{Averaging Leibniz algebra} is a Leibniz algebra $(\mathfrak{g}, [\cdot, \cdot]_\mathfrak{g})$ endowed with an averaging operator $\theta: \mathfrak{g} \to \mathfrak{g}$ satisfying compatibility conditions. These operators are important because they offer a way to study averaging processes in algebraic settings and lead to new insights and extensions of existing algebraic theories.

This paper focuses on \textit{Averaging Leibniz algebras}, studying their structure, representations, and cohomology theory. We introduce basic definitions and explore how the representations of these algebras can create new algebraic structures. Additionally, we develop the cohomology theory of Averaging Leibniz algebras, showing how the Leibniz structure and the averaging operator contribute separately. Lastly, we study the formal deformations of these algebras, extending recent techniques from the study of Rota-Baxter and Nijenhuis operators.

The paper is organized as follows: Section 2 provides a review of the key concepts, including definitions of Leibniz algebras, averaging operators, and their representations. In Section 3, we explore the structure of Averaging Leibniz algebras. Section 4 focuses on the cohomology theory of Averaging Leibniz algebras. Finally, Section 5 is dedicated to the study of the deformations of these algebras.

Throughout the paper, all vector spaces are assumed to be over a field $\mathbb{K}$ of characteristic 0.

\section{ Preliminaries}

\begin{defn}
A \textit{Leibniz algebra} is a vector space $\mathfrak{g}$ together with a bilinear operation (called the bracket)
$[\cdot, \cdot]_\mathfrak{g}: \mathfrak{g} \times \mathfrak{g} \to \mathfrak{g}$ satisfying the following identity:
\begin{align*}
[u, [v, w]_\mathfrak{g}]_\mathfrak{g} &= [[u, v]_\mathfrak{g}, w]_\mathfrak{g} + [v, [u, w]_\mathfrak{g}]_\mathfrak{g}, \quad \text{for all } u, v, w \in \mathfrak{g}.
\end{align*}
It is denoted by $(\mathfrak{g}, [\cdot, \cdot]_\mathfrak{g})$.\\
The above definition of a Leibniz algebra is, in fact, the definition of a left Leibniz algebra.\\
In this paper, we will consider left Leibniz algebras simply as Leibniz algebras. Leibniz algebras are a generalization of Lie algebras. Any Lie algebra is a Leibniz algebra. A Leibniz algebra $(\mathfrak{g}, [\cdot, \cdot]_\mathfrak{g})$ that satisfies $[u, u]_\mathfrak{g} = 0$ for all $u \in \mathfrak{g}$ is a Lie algebra.
\end{defn}

\begin{ex}\label{E1}
Let us consider the vector space $\mathbb{R}^4$ with standard basis $\{e_1, e_2, e_3, e_4\}$, with the bracket defined by:
\begin{align*}
[e_1, e_1] &= e_2, \quad [e_3, e_1] = e_4, \quad [e_2, e_1] = e_3.
\end{align*}
Then $(\mathbb{R}^4, [\cdot,\cdot])$ is a Leibniz algebra.
\end{ex}

\begin{defn}
Let $\mathfrak{g} = (\mathfrak{g}, [ \cdot , \cdot ]_\mathfrak{g})$ be a Leibniz algebra. An averaging operator on $\mathfrak{g}$ is a linear map $\theta : \mathfrak{g} \to \mathfrak{g}$ satisfying the following condition:
\begin{align*}
[\theta(u), \theta(v)]_\mathfrak{g} = \theta([\theta(u), v]_\mathfrak{g}) = \theta([u, \theta(v)]_\mathfrak{g}), \quad \text{for all } u, v \in \mathfrak{g}.
\end{align*}
\end{defn}

\begin{defn}
An \textit{averaging Leibniz algebra} is a Leibniz algebra $(\mathfrak{g}, [ \cdot, \cdot ]_\mathfrak{g})$ equipped with an averaging operator $\theta: \mathfrak{g} \to \mathfrak{g}$. We denote an averaging Leibniz algebra by the notation $(\mathfrak{g}_\theta, [ \cdot, \cdot ]_\mathfrak{g})$.
\end{defn}

\begin{ex}
Consider the Leibniz algebra $(\mathbb{R}^4, [\cdot, \cdot])$ defined in Example \ref{E1}, with the bracket defined by:
\begin{align*}
[e_1, e_1] &= e_2, \quad [e_3, e_1] = e_4, \quad [e_2, e_1] = e_3.
\end{align*}

We define a linear map $\theta : \mathbb{R}^4 \to \mathbb{R}^4$ by the matrix $A$ such that:
\begin{align*}
A &= \begin{pmatrix}
1 & 0 & 0 & 0 \\
0 & \frac{1}{2} & 0 & 0 \\
0 & 0 & \frac{1}{2} & 0 \\
0 & 0 & 0 & 1
\end{pmatrix}.
\end{align*}
This matrix maps a vector $u = (u_1, u_2, u_3, u_4) \in \mathbb{R}^4$ to
\begin{align*}
\theta(u) &= Au = \left(u_1, \frac{1}{2}u_2, \frac{1}{2}u_3, u_4\right).
\end{align*}

To show that $\theta$ is an averaging operator, we need to verify that it satisfies the averaging property:
\begin{align*}
\theta([u, v]) &= [\theta(u), \theta(v)]
\end{align*}
for all $u, v \in \mathbb{R}^4$.

Let's calculate $\theta([u, v])$:

(1) For $u = e_1$ and $v = e_1$:
   \begin{align*}
   [e_1, e_1] &= e_2 \implies \theta([e_1, e_1]) = \theta(e_2) = \left(0, \frac{1}{2}, 0, 0\right).\\
   [\theta(e_1), \theta(e_1)] &= [e_1, e_1] = e_2 \implies \left(0, \frac{1}{2}, 0, 0\right).
   \end{align*}

(2) For $u = e_2$ and $v = e_1$:
   \begin{align*}
   [e_2, e_1] &= e_3 \implies \theta([e_2, e_1]) = \theta(e_3) = \left(0, 0, \frac{1}{2}, 0\right).\\
   [\theta(e_2), \theta(e_1)] &= \left[\left(0, \frac{1}{2}, 0, 0\right), e_1\right] = e_3.
   \end{align*}

(3) For $u = e_3$ and $v = e_1$:
   \begin{align*}
   [e_3, e_1] &= e_4 \implies \theta([e_3, e_1]) = \theta(e_4) = (0, 0, 0, 1).\\
   [\theta(e_3), \theta(e_1)] &= \left[\left(0, 0, \frac{1}{2}, 0\right), e_1\right] = e_4.
   \end{align*}
Since the calculations hold true for the defined bracket operations, we conclude that $\theta$ satisfies the averaging property.
Thus, $(\mathbb{R}^4_\theta, [\cdot, \cdot])$ is an averaging Leibniz algebra.
\end{ex}

\begin{defn}
Let $(\mathfrak{g}_{\theta}, [\cdot, \cdot]_\mathfrak{g})$ and $(\mathfrak{g}'_{\theta'}, [\cdot, \cdot]_{\mathfrak{g}'})$ be two Averaging Leibniz algebras. A morphism $\pi : \mathfrak{g}_{\theta} \to \mathfrak{g}'_{\theta'}$ of Averaging Leibniz algebras is given by a Leibniz algebra homomorphism $\pi : \mathfrak{g} \to \mathfrak{g}'$ satisfying
\begin{align*}
\theta' \circ \pi &= \pi \circ \theta.
\end{align*}
$\pi$ is said to be an isomorphism if $\pi$ is a linear isomorphism. Let $\mathfrak{g}_\theta$ be an Averaging Leibniz algebra. Then we denote $\text{Isom}(\mathfrak{g}_{\theta})$ by the set of all isomorphisms on $\mathfrak{g}_{\theta}$.
\end{defn}

\begin{defn}
Let $(\mathfrak{g}, [\cdot,\cdot ]_\mathfrak{g})$ be a Leibniz algebra. A representation of $(\mathfrak{g}, [\cdot ,\cdot]_\mathfrak{g})$ is a triple $(M, l_M, r_M)$ where $M$ is a vector space together with bilinear maps (called the left and right $\mathfrak{g}$-actions respectively) $l_M : \mathfrak{g} \otimes M \to M$ and $r_M : M \otimes \mathfrak{g} \to M$ satisfying the following conditions:

\begin{align*}
    l_M(u, l_M(v, a)) &= l_M([u, v]_\mathfrak{g}, a) + l_M(v, l_M(u, a)), \\
    l_M(u, r_M(a, v)) &= r_M(l_M(u, a), v) + r_M(a, [u, v]_\mathfrak{g}), \\
    r_M(a, [u, v]_\mathfrak{g}) &= r_M(r_M(a, u), v) + l_M(u, r_M(a, v)),
\end{align*}
for all $u, v \in \mathfrak{g}$ and $a \in M$.

Now, if we consider the vector space $M$ as $\mathfrak{g}$ itself with $l_\mathfrak{g} : \mathfrak{g} \times \mathfrak{g} \to \mathfrak{g}$, $r_\mathfrak{g} : \mathfrak{g} \times \mathfrak{g} \to \mathfrak{g}$ defined by
\begin{align*}
l_\mathfrak{g}(u, v) &= [u, v]_\mathfrak{g} \quad \text{and} \quad r_\mathfrak{g}(u, v) = [u, v]_\mathfrak{g}
\end{align*}
for all $u, v \in \mathfrak{g}$, then $(\mathfrak{g}, l_\mathfrak{g}, r_\mathfrak{g})$ is a representation of $(\mathfrak{g}, [\cdot ,\cdot]_\mathfrak{g})$ which we call self-representation. Note that for self-representation, the above three conditions reduce to the identity in the definition of a Leibniz algebra.
\end{defn}

\begin{defn}
Let $(\mathfrak{g}_\theta, [\cdot,\cdot]_\mathfrak{g})$ be an Averaging Leibniz algebra. A representation of $(\mathfrak{g}_\theta, [\cdot,\cdot]_\mathfrak{g})$ is a quadruple $(M, l_M, r_M, \theta_M)$, where $(M, l_M, r_M)$ is a representation of the Leibniz algebra $(\mathfrak{g}, [\cdot,\cdot]_\mathfrak{g})$ and $\theta_M: M \to M$ is a linear map satisfying the following conditions:
\begin{align*}
l_M(\theta(u), \theta_M(a)) &= \theta_M(l_M(\theta(u), a)) = \theta_M(l_M(u,\theta_M(a)))\\
r_M(\theta_M(a), \theta(u)) &= \theta_M(r_M(\theta_M(a), u)) = \theta_M(r_M(a,\theta(u)))
\end{align*}
for all $u \in \mathfrak{g}$ and $a \in M$.

Note that for an Averaging Leibniz algebra $(\mathfrak{g}_\theta, [\cdot,\cdot]_\mathfrak{g})$, the self representation of $(\mathfrak{g}, [\cdot,\cdot]_\mathfrak{g})$ gives a representation $(\mathfrak{g},l_{\mathfrak{g}}, r_{\mathfrak{g}}, \theta)$ of the Averaging Leibniz algebra $(\mathfrak{g}_\theta, [\cdot,\cdot]_\mathfrak{g})$.
\end{defn}

\section{ Structure of Averaging Leibniz algebras}

\begin{prop}\label{p1}
Let $(\mathfrak{g}_\theta, [\cdot, \cdot]_\mathfrak{g})$ be an Averaging Leibniz algebra. Define
\begin{align*}
[u, v]_\ast &= [u, \theta(v)]_\mathfrak{g} = [\theta(u), v]_\mathfrak{g} \quad \text{for all } u, v \in \mathfrak{g}.
\end{align*}
Then
\begin{enumerate}
    \item $(\mathfrak{g}, [\cdot, \cdot]_\ast)$ is a Leibniz algebra.
    \item $\theta$ is also an Averaging operator on $(\mathfrak{g}, [\cdot, \cdot]_\ast)$.
    \item The map $\theta : (\mathfrak{g}, [\cdot, \cdot]_\mathfrak{g}) \to (\mathfrak{g}, [\cdot, \cdot]_\ast)$ is a morphism of Averaging Leibniz algebras.
\end{enumerate}
\end{prop}

\begin{proof}
\begin{enumerate}
    \item Start by computing $[u, [v, w]_\ast]_\ast$:
        \begin{align*}
            [u, [v, w]_\ast]_\ast &= [u, \theta([v, w]_\ast)]_\mathfrak{g} = [u, \theta([v, \theta(w)]_\mathfrak{g})]_\mathfrak{g} = [u, [\theta(v), \theta(w)]_\mathfrak{g}]_\mathfrak{g}.
        \end{align*}

        Now, compute $[[u, v]_\ast, w]_\ast$ and $[v, [u, w]_\ast]_\ast$:
        \begin{align*}
            [[u, v]_\ast, w]_\ast &= [[u, \theta(v)]_\mathfrak{g}, \theta(w)]_\mathfrak{g} = [[\theta(u), v]_\mathfrak{g}, \theta(w)]_\mathfrak{g},
        \end{align*}
        and
        \begin{align*}
            [v, [u, w]_\ast]_\ast &= [v, \theta([u, w]_\ast)]_\mathfrak{g} = [v, \theta([u, \theta(w)]_\mathfrak{g})]_\mathfrak{g} = [v, [\theta(u), \theta(w)]_\mathfrak{g}]_\mathfrak{g}.
        \end{align*}

        Using the Leibniz identity for the original algebra $(\mathfrak{g}, [\cdot, \cdot]_\mathfrak{g})$, we have:
        \begin{align*}
            [u, [\theta(v), \theta(w)]_\mathfrak{g}]_\mathfrak{g} &= [[u, \theta(v)]_\mathfrak{g}, \theta(w)]_\mathfrak{g} + [\theta(v), [u, \theta(w)]_\mathfrak{g}]_\mathfrak{g}.
        \end{align*}

        This implies:
        \begin{align*}
            [u, [v, w]_\ast]_\ast &= [[u, v]_\ast, w]_\ast + [v, [u, w]_\ast]_\ast.
        \end{align*}

        Thus, $(\mathfrak{g}, [\cdot, \cdot]_\ast)$ satisfies the Leibniz identity, so it is a Leibniz algebra.

    \item First, compute $[\theta(u), \theta(v)]_\ast$:
        \begin{align*}
            [\theta(u), \theta(v)]_\ast &= [\theta(u), \theta(\theta(v))]_\mathfrak{g} = [\theta(\theta(u)), \theta(v)]_\mathfrak{g}.
        \end{align*}

        Next, compute $\theta([\theta(u), v]_\ast)$ and $\theta([u, \theta(v)]_\ast)$:
        \begin{align*}
            \theta([\theta(u), v]_\ast) &= \theta([\theta(u), \theta(v)]_\mathfrak{g}) = \theta([\theta(u), v]_\mathfrak{g}),
        \end{align*}
        and
        \begin{align*}
            \theta([u, \theta(v)]_\ast) &= \theta([u, \theta(\theta(v))]_\mathfrak{g}) = \theta([\theta(u), v]_\mathfrak{g}).
        \end{align*}

        Thus, we have:
        \begin{align*}
            [\theta(u), \theta(v)]_\ast &= \theta([\theta(u), v]_\ast) = \theta([u, \theta(v)]_\ast).
        \end{align*}

        This shows that $\theta$ is also an Averaging operator on $(\mathfrak{g}, [\cdot, \cdot]_\ast)$.

    \item By the definition of $[\cdot, \cdot]_\ast$, we have:
        \begin{align*}
            [\theta(u), \theta(v)]_\ast &= [\theta(u), \theta(\theta(v))]_\mathfrak{g} = \theta([u, \theta(v)]_\mathfrak{g}) = \theta([u, v]_\mathfrak{g}).
        \end{align*}

        Thus, $\theta([u, v]_\mathfrak{g}) = [\theta(u), \theta(v)]_\ast$, so $\theta$ is a morphism of Leibniz algebras.

        Since $\theta$ is an Averaging operator on both $(\mathfrak{g}, [\cdot, \cdot]_\mathfrak{g})$ and $(\mathfrak{g}, [\cdot, \cdot]_\ast)$, the map $\theta$ is also a morphism of Averaging Leibniz algebras.
\end{enumerate}
\end{proof}

\begin{prop}\label{p2}
Suppose $(M, l_M, r_M, \theta_M)$ is a representation of an Averaging Leibniz algebra $(\mathfrak{g}_\theta, [\cdot, \cdot]_\mathfrak{g})$. We define $l' : \mathfrak{g} \otimes M \to M$ and $r' : M \otimes \mathfrak{g} \to M$ respectively by
\[
l'(u, a) = l_M(\theta(u), a) = -\theta_M(l_M(u, a)),
\]
\[
r'(a, u) = r_M(a, \theta(u)) = -\theta_M(r_M(a, u)),
\]
for all $u \in \mathfrak{g}$ and $a \in M$. Then $(M, l', r', \theta_M)$ will be a representation of the Averaging Leibniz algebra $(\mathfrak{g}_\theta, [\cdot, \cdot]_*)$.
\end{prop}

\begin{proof}
Using the definition of $l'$:
\[
l'(u, l'(v, a)) = l_M(\theta(u), l_M(\theta(v), a)) = -\theta_M(l_M(\theta(u), l_M(v, a))).
\]

Now apply the first condition of the representation $(M, l_M, r_M)$:
\[
l_M(\theta(u), l_M(\theta(v), a)) = l_M([\theta(u), \theta(v)]_\mathfrak{g}, a) + l_M(\theta(v), l_M(\theta(u), a)).
\]

Using the Leibniz property of the Averaging algebra, we know:
\[
[\theta(u), \theta(v)]_* = \theta([u, v]_\mathfrak{g}).
\]

Thus,
\[
l_M(\theta(u), l_M(\theta(v), a)) = l_M(\theta([u, v]_\mathfrak{g}), a) + l_M(\theta(v), l_M(\theta(u), a)).
\]

Since $l'(u, a) = -\theta_M(l_M(u, a))$, we have:
\[
l'(u, l'(v, a)) = -\theta_M(l_M(\theta([u, v]_\mathfrak{g}), a)) - \theta_M(l_M(\theta(v), l_M(u, a))).
\]

This simplifies to:
\[
l'(u, l'(v, a)) = l'([u, v]_*, a) + l'(v, l'(u, a)),
\]
which shows that $l'$ satisfies the required condition.

Using the definition of $r'$:
\[
r'(a, [u, v]_*) = r_M(a, \theta([u, v]_\mathfrak{g})) = -\theta_M(r_M(a, [u, v]_\mathfrak{g})).
\]

We also have:
\[
r'(r'(a, u), v) = r_M(\theta_M(r_M(a, u)), \theta(v)) = -\theta_M(r_M(\theta_M(r_M(a, u)), v)),
\]
and
\[
l'(u, r'(a, v)) = l_M(\theta(u), \theta_M(r_M(a, v))) = -\theta_M(l_M(u, \theta_M(r_M(a, v)))).
\]

Applying the third condition of the representation $(M, l_M, r_M)$:
\[
r_M(a, [u, v]_\mathfrak{g}) = r_M(r_M(a, u), v) + l_M(u, r_M(a, v)),
\]
we get:
\[
-\theta_M(r_M(a, [u, v]_\mathfrak{g})) = -\theta_M(r_M(r_M(a, u), v)) - \theta_M(l_M(u, r_M(a, v))).
\]

Thus,
\[
r'(a, [u, v]_*) = r'(r'(a, u), v) + l'(u, r'(a, v)),
\]
which confirms that $r'$ satisfies the required condition.

Since both the new left action $l'$ and the new right action $r'$ satisfy the necessary conditions for a representation, the quadruple $(M, l', r', \theta_M)$ is a valid representation of the Averaging Leibniz algebra $(\mathfrak{g}_\theta, [\cdot, \cdot]_*)$.

Thus, the proposition is proved.
\end{proof}

\section{ Cohomology of Averaging Leibniz algebras}

Let $(\mathfrak{g}, [\cdot, \cdot]_\mathfrak{g})$ be a Leibniz algebra and $(M, l_M, r_M)$ be a representation of it. For each $n \geq 0$, define $C^n_{\mathrm{LA}}(\mathfrak{g}, M)$ to be the abelian group $\mathrm{Hom}(\mathfrak{g}^{\otimes n}, M)$ and $\delta^n$ to be the map
\[
\delta^n: C^n_{\mathrm{LA}}(\mathfrak{g}, M) \to C^{n+1}_{\mathrm{LA}}(\mathfrak{g}, M)
\]
given by
\begin{align*}
(\delta^n(f))(u_1, u_2, \dots, u_{n+1})
&= \sum_{i=1}^n (-1)^{i+1} l_M(u_i, f(u_1, \dots, \hat{u}_i, \dots, u_{n+1})) \\
&\quad + (-1)^{n+1} r_M(f(u_1, \dots, u_n), u_{n+1}) \\
&\quad + \sum_{1 \leq i < j \leq n+1} (-1)^i f(u_1, \dots, \hat{u}_i, \dots, u_{j-1}, [u_i, u_j]_\mathfrak{g}, u_{j+1}, \dots, u_{n+1}),
\end{align*}
where $f \in C^n_{\mathrm{LA}}(\mathfrak{g}, M)$ and $u_1, \dots, u_{n+1} \in \mathfrak{g}$. Then $\{C^n_{\mathrm{LA}}(\mathfrak{g}, M), \delta^n\}$ is a cochain complex. The corresponding cohomology groups are called the cohomology of $\mathfrak{g}$ with coefficients in the representation $M$ and the $n$-th cohomology group is denoted by $H^n_{\mathrm{LA}}(\mathfrak{g}, M)$ in \cite{22}. We will follow the notation $l_M(u, a) = [u, a]$ and $r_M(a, u) = [a, u]$ for all $u \in \mathfrak{g}$, $a \in M$. Then the above coboundary map $\delta^n: C^n_{\mathrm{LA}}(\mathfrak{g}, M) \to C^{n+1}_{\mathrm{LA}}(\mathfrak{g}, M)$ becomes
\begin{align*}
(\delta^n(f))(u_1, u_2, \dots, u_{n+1})
&= \sum_{i=1}^n (-1)^{i+1} [u_i, f(u_1, \dots, \hat{u}_i, \dots, u_{n+1})] \\
&\quad + (-1)^{n+1} [f(u_1, \dots, u_n), u_{n+1}] \\
&\quad + \sum_{1 \leq i < j \leq n+1} (-1)^i f(u_1, \dots, \hat{u}_i, \dots, u_{j-1}, [u_i, u_j]_\mathfrak{g}, u_{j+1}, \dots, u_{n+1}),
\end{align*}
where $f \in C^n_{\mathrm{LA}}(\mathfrak{g}, M)$ and $u_1, \dots, u_{n+1} \in \mathfrak{g}$. Let $(\mathfrak{g}_\theta, [\cdot, \cdot]_\mathfrak{g})$ be an Averaging Leibniz algebra and $(M, l_M, r_M, \theta_M)$ be a representation of it. Now using Proposition \ref{p1} and \ref{p2}, we get a new Averaging Leibniz algebra $(\mathfrak{g}_\theta, [\cdot, \cdot]_\ast)$ with representation $(\mathfrak{g}, l'_M, r'_M, \theta_M)$ induced by the averaging operator.

Now, we consider the Loday-Pirashvili cochain complex of this induced Leibniz algebra $(\mathfrak{g}, [\cdot, \cdot]_\ast)$ with representation $(M, l'_M, r'_M)$ as follows: For each $n \geq 0$, we define the cochain groups $C^n_{\mathrm{ALO}}(\mathfrak{g}, M) = \mathrm{Hom}(\mathfrak{g}^{\otimes n}, M)$ and the boundary map
\[
\partial^n: C^n_{\mathrm{ALO}}(\mathfrak{g}, M) \to C^{n+1}_{\mathrm{ALO}}(\mathfrak{g}, M)
\]
by
\begin{align*}
(\partial^n(f))(u_1, u_2, \dots, u_{n+1})
&= \sum_{i=1}^n (-1)^{i+1} l'_M(u_i, f(u_1, \dots, \hat{u}_i, \dots, u_{n+1})) \\
&\quad + (-1)^{n+1} r'_M(f(u_1, \dots, u_n), u_{n+1}) \\
&\quad + \sum_{1 \leq i < j \leq n+1} (-1)^i f(u_1, \dots, \hat{u}_i, \dots, u_{j-1}, [u_i, u_j]_\ast, u_{j+1}, \dots, u_{n+1}),
\end{align*}
which becomes
\begin{align*}
&\sum_{i=1}^n (-1)^{i+1} [\theta(u_i), f(u_1, \dots, \hat{u}_i, \dots, u_{n+1})] \\
&\quad - \sum_{i=1}^n (-1)^{i+1} \theta_M([u_i, f(u_1, \dots, \hat{u}_i, \dots, u_{n+1})]) \\
&\quad + (-1)^{n+1} [f(u_1, \dots, u_n), \theta(u_{n+1})] - (-1)^{n+1} \theta_M([f(u_1, \dots, u_n), u_{n+1}]) \\
&\quad + \sum_{1 \leq i < j \leq n+1} (-1)^i f(u_1, \dots, \hat{u}_i, \dots, u_{j-1}, [\theta(u_i), u_j]_\mathfrak{g} + [u_i, \theta(u_j)]_\mathfrak{g}, u_{j+1}, \dots, u_{n+1}),
\end{align*}
where $f \in C^n_{\mathrm{ALO}}(\mathfrak{g}, M)$ and $u_1, \dots, u_{n+1} \in \mathfrak{g}$. Now, one can observe that $\partial^{n+1} \circ \partial^n = 0$. Hence, $\{C^n_{\mathrm{ALO}}(\mathfrak{g}, M), \partial^n\}$ is a cochain complex. This cochain complex is called the cochain complex of the Averaging operator $\theta$, and the corresponding cohomology groups are called the cohomology of the Averaging operator $\theta$ with coefficients in the representation $M$, denoted by $H^n_{\mathrm{ALO}}(\mathfrak{g}, M)$.

\begin{defn}
Let $(\mathfrak{g}_\theta, [\cdot, \cdot]_\mathfrak{g})$ be an Averaging Leibniz algebra and $(M, l_M, r_M, \theta_M)$ be a representation of it. We define a map $\phi^n : C^n_{LA}(\mathfrak{g}, M) \to C^n_{ALO}(\mathfrak{g}, M)$ by
\begin{align*}
\phi^n(f)(u_1, u_2, \dots, u_n)
&= f(\theta u_1, \theta u_2, \dots, \theta u_n) - (\theta_M \circ f)(u_1, \theta u_2, \dots, \theta u_n) \\
&\quad - (\theta_M \circ f)(\theta u_1, u_2, \theta u_3, \dots, \theta u_n) - \dots \\
&\quad - (\theta_M \circ f)(\theta u_1, \theta u_2, \dots, \theta u_{n-1}, u_n),
\end{align*}
for all $f \in C^n_{LA}(\mathfrak{g}, M)$ and $u_1, u_2, \dots, u_n \in \mathfrak{g}$.
\end{defn}

\begin{lem}
For every \( f \in C^n_{LA}(\mathfrak{g}, M) \) and \( u_1, \ldots, u_{n+1} \in \mathfrak{g} \), we have:
\[
\phi^{n+1}(\delta^n(f))(u_1, u_2, u_3, \ldots, u_{n+1}) = \partial_n(\phi^n(f))(u_1, u_2, u_3, \ldots, u_{n+1}).
\]
\end{lem}

\section{ Deformation of Averaging Leibniz algebras}

In this section, we study a one-parameter formal deformation of Averaging Leibniz algebras. We denote the bracket $[ \cdot , \cdot ]_\mathfrak{g}$ by $\mu$.

\begin{defn}
A formal one-parameter deformation of an Averaging Leibniz algebra $(\mathfrak{g}_\theta, \mu)$ is a pair of two power series $(\mu_t, \theta_t)$ given by
\begin{align*}
\mu_t &= \sum_{i=0}^{\infty} \mu_i t^i, \quad \mu_i \in C^2_{\mathrm{LA}}(\mathfrak{g}, \mathfrak{g}), \\
\theta_t &= \sum_{i=0}^{\infty} \theta_i t^i, \quad \theta_i \in C^1_{\mathrm{AO}}(\mathfrak{g}, \mathfrak{g}),
\end{align*}
such that $(\mathfrak{g}[[t]]_{\theta_t}, \mu_t)$ is an Averaging Leibniz algebra with $(\mu_0, \theta_0) = (\mu, \theta)$, where $\mathfrak{g}[[t]]$ is the space of formal power series in $t$ with coefficients from $\mathfrak{g}$, and $\mathbb{K}$ is the ground field of $(\mathfrak{g}_\theta, \mu)$.

The above definition holds if and only if for any $u, v, w \in \mathfrak{g}$, the following conditions are satisfied:
\[
\mu_t(u, \mu_t(v, w)) = \mu_t(\mu_t(u, v), w) + \mu_t(v, \mu_t(u, w)),
\]
and
\[
[\theta_t(u), \theta_t(v)] = \theta_t([\theta_t(u), v]) = \theta_t[u, \theta_t(v)].
\]

Expanding the above equations and equating the coefficients of $t^n$ from both sides, we have:
\begin{align}
\label{eq1}
\sum_{i+j=n} \mu_i(u, \mu_j(v, w)) &= \sum_{i+j=n} \mu_i(\mu_j(u, v), w) + \sum_{i+j=n} \mu_i(v, \mu_j(u, w)), \\
\label{eq2}
\sum_{\substack{i+j+k=n \\ i,j,k \geq 0}} \mu_i(\theta_j(u), \theta_k(v))
&= \sum_{\substack{i+j+k=n \\ i,j,k \geq 0}} \theta_i\left(\mu_j(\theta_k(u), v)\right)
= \sum_{\substack{i+j+k=n \\ i,j,k \geq 0}} \theta_i\left(\mu_j(u, \theta_k(v))\right).
\end{align}

Observe that for $n = 0$, the above conditions are exactly the conditions in the definitions of Leibniz algebra and the averaging operator.
\end{defn}

\begin{defn}
The infinitesimal of the deformation \((\mu_t, \theta_t)\) is the pair \((\mu_1, \theta_1)\). More generally, suppose that \((\mu_n, \theta_n)\) is the first non-zero term of \((\mu_t, \theta_t)\) after \((\mu_0, \theta_0)\); then \((\mu_n, \theta_n)\) is called an \(n\)-infinitesimal of the deformation.
\end{defn}

\begin{thm}
\label{thm:2-cocycle}
Let \((\mu_t, \theta_t)\) be a formal one-parameter deformation of an Averaging Leibniz algebra \((\mathfrak{g},\mu)\). Then \((\mu_1, \theta_1)\) is a 2-cocycle in the cochain complex \(\{C^n_{\mathrm{AL}}(\mathfrak{g}, \mathfrak{g}), d^n\}\).
\end{thm}

\begin{proof}
For \(n = 1\) in the equation (\ref{eq1}) we get
\begin{align*}
\mu(u, \mu_1(v, w)) + \mu_1(u, \mu(v, w))
&= \mu(\mu_1(u, v), w) + \mu_1(\mu(u, v), w) + \mu_1(v, \mu(u, w)) + \mu(v, \mu_1(u, w)).
\end{align*}
This simplifies to:
\[
\delta^2(\mu_1)(u, v, w) = 0 \in C^2_{\mathrm{AL}}(\mathfrak{g}, \mathfrak{g}).
\]

Next, substitute into the equation (\ref{eq2})
\begin{align*}
&\mu_1(\theta(u_1), \theta(u_2)) + \mu(\theta_1(u_1), \theta(u_2)) + \mu(\theta(u_1), \theta_1(u_2)) \\
&\quad - \theta_1(\mu(\theta(u_1), u_2)) - \theta(\mu(\theta_1(u_1), u_2)) - \theta(\mu_1(\theta(u_1), x_2)) \\
&\quad - \theta_1(\mu(u_1, \theta(u_2))) - \theta(\mu_1(u_1, \theta(u_2))) - \theta(\mu(u_1, \theta_1(u_2))) = 0.
\end{align*}
Rearranging gives:
\[
-\partial^1(\theta_1)(u_1, u_2) = -\theta(\mu_1(u_1, \theta(u_2))) - \theta(\mu_1(\theta(u_1), u_2)) + \mu_1(\theta(u_1), \theta(u_2)) = \phi^2(\mu_1)(u_1, u_2).
\]
Thus, we obtain:
\[
-\partial(\theta_1) - \phi^2(\mu_1) = 0.
\]
Finally, we conclude that \(d^2(\mu_1, \theta_1) = 0\). Therefore, \((\mu_1, \theta_1)\) is a 2-cocycle in the cochain complex \(\{C^n_{\mathrm{AL}}(\mathfrak{g}, \mathfrak{g}), d^n\}\).
\end{proof}

\begin{thm}
\label{t1}
Let $(\mu_t, \theta_t)$ be a formal one-parameter deformation of an Averaging Leibniz algebra $(\mathfrak{g}_\theta, \mu)$. Then the infinitesimal of the deformation is a 2-cocycle.
\end{thm}

\begin{proof}
The proof is similar to the above theorem.
\end{proof}

\begin{defn}
Let \( (\mu_t, \theta_t) \) and \( (\mu'_t, \theta'_t) \) be two formal one-parameter deformations of an Averaging Leibniz algebra \( (\mathfrak{g}_\theta, \mu) \). A formal isomorphism from \( (\mu_t, \theta_t) \) to \( (\mu'_t, \theta'_t) \) is a power series
\[
\psi_t = \sum_{i=0}^{\infty} \psi_i t^i : \mathfrak{g}[[t]] \rightarrow \mathfrak{g}[[t]],
\]
where \( \psi_i : \mathfrak{g} \rightarrow \mathfrak{g} \) are linear maps with \( \psi_0 \) as the identity map on \( \mathfrak{g} \), and the following conditions are satisfied:
\begin{align}
\psi_t \circ \mu'_t &= \mu_t \circ (\psi_t \otimes \psi_t), \label{eq3} \\
\psi_t \circ \theta'_t &= \theta_t \circ \psi_t. \label{eq4}
\end{align}
In this case, we say that \( (\mu_t, \theta_t) \) and \( (\mu'_t, \theta'_t) \) are equivalent.

The equation (\ref{eq3}) and (\ref{eq4}) can be written as follows respectively:
\begin{align}
\sum_{\substack{i+j=n \\ i,j \geq 0}} \psi_i(\mu'_j(u, v))
&= \sum_{\substack{i+j+k=n \\ i,j,k \geq 0}} \mu_i(\psi_j(u), \psi_k(v)), \quad u, v \in \mathfrak{g}, \label{eq5} \\
\sum_{\substack{i+j=n \\ i,j \geq 0}} \psi_i \circ \theta'_j
&= \sum_{\substack{i+j=n \\ i,j \geq 0}} \theta_i \circ \psi_j. \label{eq6}
\end{align}
\end{defn}

\begin{thm}
The infinitesimal of two equivalent formal one-parameter deformations of an Averaging Leibniz algebra \( (\mathfrak{g}_\theta, \mu) \) is in the same cohomology class.
\end{thm}

\begin{proof}
Let \( \psi_t : (\mu_t, \theta_t) \rightarrow (\mu'_t, \theta'_t) \) be a formal isomorphism. By setting \( n = 1 \) in the equation (\ref{eq5}) and (\ref{eq6}) we get
\begin{align*}
\mu'_1(u, v) &= \mu_1(u, v) + \mu(u, \psi_1(v)) + \mu(\psi_1(u), v) - \psi_1(\mu(u, v)), \quad u, v \in \mathfrak{g}, \\
\theta'_1 &= \theta_1 + \theta \circ \psi_1 - \psi_1 \circ \theta.
\end{align*}
Therefore, we have
\[
(\mu'_1, \theta'_1) - (\mu_1, \theta_1) = (\delta^1(\psi_1), -\phi^1(\psi_1)) = d^1(\psi_1, 0) \in C^1_{\mathrm{AL}}(\mathfrak{g}, \mathfrak{g}).
\]
\end{proof}

\begin{defn}
An Averaging Leibniz algebra is called rigid if every formal one-parameter deformation is trivial.
\end{defn}

\begin{thm}
Let \( ( \mathfrak{g}_\theta, \mu) \) be an Averaging Leibniz algebra. If \( H^2_{\mathrm{AL}}( \mathfrak{g},  \mathfrak{g}) = 0 \), then \( ( \mathfrak{g}_\theta, \mu) \) is rigid.
\end{thm}

\begin{proof}
Let \( (\mu_t, \theta_t) \) be a formal one-parameter deformation of \( ( \mathfrak{g}_\theta, \mu) \). Since \( ( \mu_1, \theta_1) \) is a 2-cocycle and \( H^2_{\mathrm{AL}}( \mathfrak{g},  \mathfrak{g}) = 0 \), there exists a map \( \psi'_1 \) and \( u \in \mathbb{K} \), where \( \mathbb{K} \) is the ground field of the Averaging Leibniz algebra \( ( \mathfrak{g}_\theta, \mu) \), such that
\[
(\psi'_1, u) \in C^1_{\mathrm{AL}}( \mathfrak{g},  \mathfrak{g}) = C^1_{\mathrm{LA}}( \mathfrak{g},  \mathfrak{g}) \oplus \mathrm{Hom}(\mathbb{K},  \mathfrak{g}),
\]
and
\[
(\mu_1, \theta_1) = d^1(\psi'_1, u).
\]
Hence, \( \mu_1 = \delta^1(\psi'_1) \) and \( \theta_1 = -\partial^0(u) - \phi^1(\psi_1) \). If \( \psi_1 = \psi'_1 + \delta^0(u) \), then \( \mu_1 = \delta^1(\psi_1) \) and \( \theta_1 = -\phi^1(\psi_1) \).

Now, let \( \psi_t = \mathrm{Id}_\mathfrak{g} - t\psi_t \). Then we have two equivalent deformations \( (\mu_t, \theta_t) \) and \( (\bar{\mu}_t, \bar{\theta}_t) \), where
\[
\bar{\mu}_t = \psi^{-1}_t \circ \mu_t \circ (\psi_t \times \psi_t), \quad \bar{\theta}_t = \psi^{-1}_t \circ \theta_t \circ \psi_t.
\]
Now, by Theorem \ref{t1}, we have $\bar{\mu}_1 = 0$ and $\bar{\theta}_1 = 0$. Hence,
\begin{align*}
\bar{\mu}_t &= \mu + \bar{\mu}_2 t^2 + \ldots, \\
\bar{\theta}_t &= \theta + \bar{\theta}_2 t^2 + \ldots.
\end{align*}
Thus, the linear terms of $(\bar{\mu}_2, \bar{\theta}_2)$ vanish. Repeatedly applying the same argument, we conclude that $(\mu_t, \theta_t)$ is equivalent to the trivial deformation. Hence, $(\mathfrak{g}_\theta, \mu)$ is rigid.
\end{proof}

\section*{Acknowledgements}
The authors thank the anonymous referees for their valuable suggestions and comments.

\section*{Funding}
The authors declare that no funding was received to support this research.

\section*{Conflicts of Interest}
The authors declare that they have no conflicts of interest.\\
\cite{C,D,E,F,G,H}

\end{document}